\documentclass[12pt]{extarticle}
\usepackage[colorlinks=true, allcolors=blue]{hyperref}
\usepackage{bm,amssymb,amsmath,amsthm,microtype}
\usepackage[total={148mm,193mm}]{geometry}
\usepackage[bitstream-charter]{mathdesign}

\author{Johan Andersson\thanks{Email:johan.andersson@oru.se \, Address: Department of Mathematics, School of Science and Technology, {\"O}rebro University, {\"O}rebro, SE-701 82 Sweden. } }
 
\title{Polynomial approximation avoiding values in  sets II}
\theoremstyle{plain} 
\newtheorem{thm}{Theorem}  
\newtheorem{lem}{Lemma}
\newtheorem{prop}{Proposition}
\newtheorem{cor}{Corollary} 
\theoremstyle{definition}

\newtheorem{rem}{Remark}
\date{}
\newcommand{\norm}[1]{\left \Vert {#1} \right \Vert}
\newcommand{\vvx}{x}

\def\cprime{$'$}

\newcommand{\C}{{\mathbb C}}

\newcommand{\R}{{\mathbb R}}
\newcommand{\abs}[1]{{\left| {#1} \right|}}
\renewcommand{\emptyset}{\varnothing}
\newcommand{\p}[1]{\left( {#1} \right)}
\begin{document}

\maketitle  
      
\begin{abstract} 
 We prove some results on when functions on compact sets $K  \subset \mathbb C$ can be approximated by polynomials avoiding values in given sets. We also prove some higher dimensional analogues. In  particular we prove that a continuous function from a compact set $K \subset \mathbb R^n$ without interior points to $\mathbb R^n$ can be uniformly approximated by a polynomial mapping avoiding values in any given countable set $A \subset \mathbb R^n$, giving a real $n$-dimensional analogue of  a recent version of Lavrentiev's theorem of Andersson and Rousu. We also prove the same result for infinite dimensional Banach spaces.
\end{abstract}
   
\maketitle
\section{Introduction}
In \cite{AndRos}, \cite{Rousu} the following version of Lavrent\'ev's theorem was proved
 \begin{thm} \label{TH1}
  Let $A \subset \C $ be any countable set, let $K \subset \C$ be  a compact set with connected complement and without interior points, and let $f$ be a continuous function on $K$. Then  given any $\varepsilon>0$ there exists some polynomial $p$ such that $p(z) \not \in A$ if $z \in K$ and such that
  \[
    \max_{z \in K} \abs{f(z)-p(z)}<\varepsilon.
  \]
 \end{thm}
 Lavrentiev's theorem \cite{Lavrentiev} is the special case where $A=\emptyset$ and Theorem \ref{TH1}  also generalize  \cite[Theorem 1]{Andersson}\footnote{A result which has applications to the Riemann zeta-function.} where it was proved for $A=\{0\}$.
 Examples of sets $K \subset \C$ with positive two dimensional Lebesgue measure where the result holds true can be given by e.g. $[0,1]+i S$ where $S \subset[0,1]$ is a fat Cantor set. In the case when the compact set $K$ has  measure zero\footnote{in which case it automatically has an empty interior} we prove this in a simpler way than in \cite{AndRos},\cite{Rousu} as a consequence of the fact that the difference set\footnote{where the difference set   $A-B=\{a-b: a \in A, b \in B\}$ here adheres to standard terminology in additive combinatorics.}   $p(K)-A$ must then also have measure zero so that its complement $(p(K)-A)^\complement$ is dense in $\C$. The general idea  will also allow us to generalize this result to higher dimensions, and also allow a sharper result when the set $K$ does not have full dimension. Finally we will adapt the method of \cite{AndRos}, \cite{Rousu} which together with the observation that if $p:\R^n \to \R^n$ is a polynomial mapping then the fact that $K$ is a compact set without interior points implies that the image $p(K)$  is also a compact set without interior points. This will allow us to prove a real $n-$dimensional analogue of Theorem \ref{TH1}.
 
 \section{Sumsets, translates of sets avoiding sets and polynomial approximation}
 \subsection{Some notation}
We use the following standard notation for sumsets, difference sets and sums of an element and a set 
 $$A \pm B=\{a \pm b: a \in A, b \in B\} , \qquad c \pm A=\{c \pm a: a \in A\},$$
 where $c$ lies in some vector space and $A,B$ are subsets of said vector space.
\subsection{A Lebesgue measure argument}
In this and the next subsection we give some conditions (Propositions \ref{prop1}, \ref{prop2}) on when translates of sets avoids sets each of which gives a corresponding result on polynomial approximation avoiding sets (Theorems \ref{TH2}, \ref{TH3}). These Propositions in turn are direct consequences of some elementary results on sumsets (or difference sets),  Lemmas  \ref{lem1},\ref{lem2}.
 \begin{lem} \label{lem1} Let $A,K \subset \R^n$, where $A$ is countable and $K$ has $n$-dimensional Lebesgue measure 0. Then the Lebesgue measure of $A\pm K$ is also 0.
 \end{lem}
\begin{proof}
 Let $A=\{a_j\}_{j=1}^\infty$ . Then
 $$A\pm K=\bigcup_{j=1}^\infty (a_j\pm K),$$ and by the subadditivity and the translation invariance of the $n$-dimensional Lebesgue measure $\mu$ we have that 
 $$\mu(A \pm K)=\mu\p{\bigcup_{j=1}^\infty (a_j\pm K)}   \leq   \sum_{j=1}^\infty \mu(a_j \pm K)=\sum_{j=1}^\infty 0 =0. $$ 
 \end{proof}
Our first Proposition which will give a simple proof of a special case of Theorem \ref{TH1} now follows as a direct consequence of Lemma \ref{lem1}.
\begin{prop} \label{prop1}
  Let $\varepsilon>0$, let $K \subset \R^n$ have $n$-dimensional Lebesgue measure 0 and let $A \subset \R^n$ be countable. Then for almost all    $\xi \in \R^n$ such that $|\xi|<\varepsilon$ then $(\xi +K) \cap A=\emptyset$. 
 \end{prop}
\begin{proof}  
By Lemma \ref{lem1} the difference set  $A-K$ has measure zero and thus also its subset $(A-K)\cap B_{\varepsilon}^\circ(0)$ also has measure zero. This means that for almost all  $|\xi|<\varepsilon$ then  $ \xi \not \in (A-K)$, which in other words means that $(\xi+K)\cap A=\emptyset$.
 \end{proof}
 Since $\C$  may be identified by $\R^2$
 we are thus ready to prove the following special case of Theorem \ref{TH1}.
\begin{thm} \label{TH2}
  Let $A \subset \C $ be any countable set, let $K \subset \C$ be  a compact set with connected complement and with measure 0, and let $f$ be a continuous function on $K$. Then  given any $\varepsilon>0$ there exists some polynomial $p$ such that $p(z) \not \in A$ if $z \in K$ and such that
  \[
    \max_{z \in K} \abs{f(z)-p(z)}<\varepsilon.
  \]
 \end{thm}
 \begin{proof}
 Since $K$ has measure zero it has no interior points and 
 by Lavrientiev's theorem there exist some polynomial $q$ such that 
\begin{gather} \label{a}
    \max_{z \in K} \abs{f(z)-q(z)}<\frac{\varepsilon} 2.
  \end{gather}
  Since $K$ has measure zero and $q$ is a polynomial, then also $q(K)$ has measure zero. 
By Proposition \ref{prop1}  we have that there exists some $\xi \in \C$ with 
 \begin{gather}
 \label{a2} 
\abs{\xi}<\frac {\varepsilon} 2,
\end{gather}
such that $(q(K)+\xi) \cap A=\emptyset$, which in other words means that $p(z):=q(z)+\xi$ does not attain values in  $A$ for $z \in K$. Our conclusion follows from \eqref{a}, \eqref{a2} and the triangle inequality.
\end{proof}
\subsection{A Hausdorff dimension argument}
 We now replace the Lebesgue measure with the Hausdorff and upper box counting dimension. The following holds true 
    \begin{lem}  \label{lem2}
       Suppose that $A,K \in \R^n$. Then
        \begin{gather} \label{aa1}
          \dim_H(A \pm K) \leq \dim_{H}(A)+\overline{\dim}_{B}(K)
        \end{gather}
         where $\dim_H$ and $\overline{\dim}_B$ denotes the Hausdorff dimension and upper box counting dimension\footnote{which coincides with the box/Minkowski dimension when it is defined \cite[Chapter 3.1]{falconer}} respectively.
         \end{lem}
         \begin{proof}
         From \cite[Product formula 7.3, p.99]{falconer} it follows that
         \begin{gather} \label{a1} \dim_H(A \times K)  \leq \dim_{H}(A)+\overline{\dim}_{B}(K). \end{gather}  
         Now let $f:\R^{2n} \to \R^n$ be given by $f(x,y)=x\pm y$ when $x,y \in  \R^n$. It is clear that $f(A\times K)=A\pm K$ and that  $f$ is Lipschitz. Thus by \cite[Corollary 2.4]{falconer}
         \begin{gather} \label{aa2} \dim_H(A  \pm K) \leq \dim_H(A \times K).
          \end{gather}
         Our conclusion follows from the inequalities \eqref{a1} and \eqref{aa2}.
       \end{proof}
 \begin{prop} \label{prop2}
  Let $\varepsilon>0$, let $K,A \in \R^n$ and let $\overline{\dim}_B (K)+\dim_H( A)<n$ or $\dim_H(K)+\overline{\dim}_B (A)<n$. Then for almost all $\xi \in \R^n$ such that $|\xi|<\varepsilon$ then $(\xi +K) \cap A=\emptyset$.
 \end{prop}
 \begin{proof}  
By Lemma \ref{lem2} the difference set  $A-K$ has dimension strictly less than $n$ and thus its $n$-dimensional Lebesgue measure is zero. It follows that also its subset $(A-K)\cap B_{\varepsilon}^\circ(0)$ has measure zero. This means that for almost all  $|\xi|<\varepsilon$ then  $ \xi \not \in (A-K)$, which in other words means that $(\xi+K)\cap A=\emptyset$.
\end{proof}
In a similar way 
 to how Proposition \ref{prop1} implies Theorem \ref{TH2}, Proposition \ref{prop2} implies the following result.
 \begin{thm} \label{TH3}
  Let $K \subset \C$ be a compact set with connected complement and let $f$ be a continuous function on  $K$ and let $A \subset \C $ be a set such that 
  \begin{gather*} \dim_H(A)+\overline{\dim}_B(K)<2, \\ \intertext{or} \overline{\dim}_B(A)+\dim_H(K)<2.
  \end{gather*}
   Then  given any $\varepsilon>0$ there exists some polynomial $p$ such that $p(z) \not \in A$ if $z \in K$ and such that
  \[
    \max_{z \in K} \abs{f(z)-p(z)}<\varepsilon.
  \]
 \end{thm}
  \begin{proof}
 Since $K$ does not have full dimension in $\C$ it has no interior points and
 by Lavrientiev's theorem there exist some polynomial $q$ such that 
\begin{gather} \label{a9}
    \max_{z \in K} \abs{f(z)-q(z)}<\frac{\varepsilon} 2.
  \end{gather}
  Since a polynomial $q:\C \to \C$ is Lipschitz we have that $\dim_H(q(K)) \leq \dim_H(K)$ and  $\overline{\dim}_B(q(K)) \leq \overline{\dim}_B(K)$. 
By Proposition \ref{prop2}  we have that there exists some $\xi \in \C$ with 
 \begin{gather}
 \label{a10} 
\abs{\xi}<\frac {\varepsilon} 2,
\end{gather}
such that $(q(K)+\xi) \cap A=\emptyset$, which in other words means that $p(z):=q(z)+\xi$ does not attain values in  $A$ for $z \in K$. Our conclusion follows from \eqref{a9}, \eqref{a10} and the triangle inequality.
\end{proof}
\subsection{A without interior points argument}
 It may be of some interest to prove a proposition on translates of sets avoiding sets corresponding to Theorem \ref{TH1}. We will do this and also show how this gives a new proof of Theorem \ref{TH1}.  Indeed we may use the methods of \cite{AndRos}, \cite{Rousu} to prove the following proposition which we may as well state and prove for a general Banach space $\mathcal B$.
  \begin{prop} \label{prop3}
  Let $\mathcal B$ be a Banach space, $\varepsilon>0, K \subset   \mathcal B$ be a compact set without interior points and let $A \subset \mathcal B$ be countable. Then for any $\varepsilon>0$ there exists some  $\xi \in \mathcal B$ with $\norm{\xi}<\varepsilon$ such that $(\xi +K) \cap A=\emptyset$. 
  \end{prop}
  We will be able to use Proposition \ref{prop3} to give a new proof of Theorem \ref{TH1}.

  \begin{proof} (We follow the proof of \cite[Lemma 2]{AndRos} with suitable modifications.)
 Let $A=\{a_j \}_{j=1}^\infty$, let $\xi_0:=0$ and let  $0<\varepsilon_0 <\varepsilon$. 
    For $j=1, 2, \ldots$ there is, since $K$ is a compact set without interior points some $\xi_j \in \mathcal B$  such that  
\begin{gather}\label{neq_a}
     \delta_j:= d(K+\xi_j,a_j)>0, 
    \end{gather}
    and such that \begin{equation} \label{appr}
\norm{\xi_j-\xi_{j-1}}<\frac{\varepsilon_{j-1}}{2},
\end{equation} 
where $\varepsilon_j>0$ for $j\geq 1$ is definied recursively so that 
\begin{equation} \label{epsj}
    \varepsilon_j < \min \left(\delta_j,\frac{\varepsilon_{j-1}} 2 \right).
\end{equation} 
By  the inequalities \eqref{appr}, \eqref{epsj}, and the triangle inequality we find if $0 \leq k <l$ that
\begin{gather}\label{Cauchy}
\begin{split}
\norm{\xi_l-\xi_k}&\leq \norm{\xi_l-\xi_{l-1}}+ \cdots +\norm{\xi_{k+1}-\xi_{k}} 
\\
&<\frac{\varepsilon_{l-1}}{2}
+\cdots+\frac{\varepsilon_k}{2}<\sum_{j=1}^{l-k}\frac{\varepsilon_k}{2^j}<\sum_{j=1}^{\infty}\frac{\varepsilon_k}{2^j}=\varepsilon_k.
\end{split} 
\end{gather} 
 By \eqref{epsj} and \eqref{Cauchy} it follows that  $\{\xi_j\}_{j=1}^\infty$ is a Cauchy-sequence in $\mathcal B$ and converges to an element   \begin{gather} \label{p} \xi:=\lim_{j \to \infty} \xi_j.
 \end{gather}
 By the definition \eqref{p} the inequality \eqref{Cauchy} implies that \begin{equation}\label{pj-p}
    \norm{\xi-\xi_j}\leq \varepsilon_j.
\end{equation}
For  $\xi$ definied by \eqref{p}, the inequalites \eqref{neq_a},  \eqref{epsj}, \eqref{pj-p}  and the triangle inequality gives us
\begin{equation} \label{yy}
    d(K+\xi,a_j) \geq d(K+\xi_j,a_j)-\norm{\xi-\xi_j} \geq \delta_j-\varepsilon_j>0,
\end{equation} for all $a_j\in A.$
  The conclusion  follows by \eqref{pj-p} and \eqref{yy}, by recalling that $ \xi_0=0$ and $0<\varepsilon_0<\varepsilon$.
 \end{proof}
   In order to give another proof of Theorem \ref{TH1} we need the following Lemma 
  \begin{lem} \label{LE3}
    Let $K \subset \C$ be a compact set without interior points and let $p$ be a polynomial. Then $p(K)$ is also a compact set without interior points.
  \end{lem}
  We might as well prove the following more general lemma in order to obtain a higher dimensional analogue of Theorem \ref{TH1} (by identifying $\C$ with $\R^2$ Lemma \ref{LE3} is a direct consequence of Lemma \ref{LE4}).
  \begin{lem} \label{LE4}
    Let $K \subset \R^n$ be a compact set without interior points and let $p$ be a polynomial mapping $p:\R^n \to \R^n$. Then $p(K)$ is also a compact set without interior points.
  \end{lem}
  
  \begin{proof}
  Since $p$ is continuous and $K$ is compact it is clear that $p(K)$ is compact. It remains to show that $p(K)$ is nowhere dense, which since $p(K)$ is compact is equivalent to the fact that $p(K)$ has no interior points.
    The result is  true if $p(z)=Az+b,$  where $b \in \R^n$ and $A$ is a non-invertible $n\times n$ matrix. In such a case $p$ maps $\R^n$ onto an affine subspace $E \subset \R^n$ of dimension at most $n-1$. Since $E$ does not have full dimension, also $p(K) \subset E$ does not have full dimension and is thus  nowhere dense.
        Now let $p$ be a polynomial mapping $p:\R^n \to \R^n$ that is not of such type. The mapping $p$ is locally invertible whenever the Jacobian determinant $\det(J_p(z))$ of $p(z)$ is non-zero.  Let $A=\{z \in K:\det(J_p(z))=0\}$ and $p(A)=B$. It is clear that $B$ is a closed set and since $p(z)$ is not of the form $Az+b$  its Jacobian determinant is a non-zero polynomial and its zero set $C \subset \R^n$ will be a real variety of dimension at most $n-1$. Since  $p$ is a polynomial mapping which is Lipschitz, also $p(C)$ and $B=p(A) \subset p(C)$ will have Hausdorff dimension at most $n-1$ and thus be nowhere dense. Let us now assume that $p(K)$ is not nowhere dense, i.e.  there exist som open set $D \subset p(K)$. Since $B$ is a nowhere dense set, we may choose some $w \in D \setminus B$ (we remark that since $D$ is open and $B$ closed, then $D\setminus B$ is open).  Let us now  consider the equation $p(z)=w$. 
         If   the equation has infinitely many roots for $z \in K$, it must have some limit point $z_0 \in K$  (since $K$ is compact) such that $p(z_0)=w$  by continuity. At such a point $z_0$ the function is not locally invertible and thus $\det(J_p(z_0))=0$ which implies that $w \in B$ contradicting our assumption that $w \in D \setminus B$. Thus we know that the equation $p(z)=w$ has  $m$ distinct roots\footnote{where by Bezouts theorem $m$ is bounded from above by $d_1 \cdots d_n$, where $d_j$ is the degree of the polynomial $p_j$ where $p=(p_1,\ldots,p_n)$} for each $j=1,\ldots,m$ and that $\det(J_p(z_j)) \neq 0$ for $j=1,\ldots,m$.  By the inverse function theorem the polynomial mapping $p$ gives homeomorphisms $f_j$ between a small neighborhood of $z_j$  for each $j=1,\ldots,m$ and a  small neighborhood of $w$. Let $\varepsilon>0$ be sufficiently small such that  $\mathcal O=B_{\varepsilon}(w)$ is a subset of $D \setminus B$ as well as all these small neighborhoods of $w$. Let $E_j=f_j^{-1}(\mathcal O)$.   It is clear that
      $$\mathcal O  \cap p(K)= \bigcup_{j=1}^m f_j(E_j \cap K)$$ is a finite union of nowhere dense sets (``nowhere dense'' is a topological property that is preserved by homeomorphisms) and thus  nowhere dense\footnote{For a proof see Proposition 7.1 in \url{https://www.ucl.ac.uk/~ucahad0/3103_handout_7.pdf}-}. This contradicts our assumption that $\mathcal O \subset D \subset p(K)$, so $p(K)$ is nowhere dense.
      \end{proof}  
  \begin{rem}
    By replacing our argument for why the set of critical values of $p$ has measure zero by Sard's theorem, Lemma \ref{LE4} hold also for $\mathcal C^1$ functions.
  \end{rem}
    
   \noindent {\em{Proof of Theorem  \ref{TH1}.}}
    Since $K \subset \C$ is a compact set with connected complement without interior points, Lavrentiev's theorem gives us a polynomial $q$ such that 
    \begin{gather} \label{a12}
    \max_{z \in K} \abs{f(z)-q(z)}<\frac{\varepsilon} 2.
  \end{gather}
  By Lemma \ref{LE3} we have that $q(K)$ is a compact set without interior points.
By Proposition \ref{prop3}  (with $\mathcal B=\C$), we have that there exists some $\xi \in \C$ with 
 \begin{gather}
 \label{a21} 
\abs{\xi}<\frac {\varepsilon} 2,
\end{gather}
such that $(q(K)+\xi) \cap A=\emptyset$, which in other words means that $p(z):=q(z)+\xi$ does not attain values in  $A$ for $z \in K$. Our conclusion follows from \eqref{a12}, \eqref{a21} and the triangle inequality. \qed

\section{Higher dimensional results}

\subsection{Finite dimensional results}
  One advantage with our approach is that it readily generalizes to several variables, by using higher dimensional analogues of the one variable approximation theorems (Weierstrass theorem, Lavrentiev's theorem).  By the $n$-dimensional version of the Weierstrass approximation theorem\footnote{which is an easy consequence of the Stone Weierstrass theorem, see e.g. \cite{mst} } we may prove 
  \begin{thm} \label{TH4}
  Let $A \subset \R^n $ be any countable set, let $K \subset \R^n$ be  a compact set without interior points, and let $f:K \to \R^n$ be a continuous function. Then  given any $\varepsilon>0$ there exists some polynomial mapping $p$ such that $p(x) \not \in A$ if $x \in K$ and such that
  \[
    \max_{\vvx \in K} \norm{f(\vvx)-p(\vvx)}<\varepsilon.
  \]
 \end{thm}
 We remark that we will use Proposition \ref{prop3} in the proof. If we instead use Proposition \ref{prop1} (which have a somewhat simpler proof) we obtain the slightly weaker version where we have to assume that $K$ has Lebesgue measure zero.
\begin{proof} (We follow the proof of Theorem \ref{TH2}) By the multivariate Weierstrass theorem we may find a polynomial mapping $q: \R^n \to \R^n$ such that
\begin{gather} \label{aaa1} \max_{\vvx \in K}\norm{q(\vvx)-f (\vvx)}<\frac {\varepsilon} 2.
\end{gather}
 Since $K \subset \R^n$ is a compact set without interior points and $q:\R^n \to \R^n$ is a polynomial map, then by Lemma \ref{LE4} also $q(K)$ is a compact set without interior points. By Proposition \ref{prop3}  we have that there exists some $\xi \in \R^n$ with 
 \begin{gather}
 \label{aaa2} 
\norm{\xi}<\frac {\varepsilon} 2,
\end{gather}
such that $(q(K)+\xi) \cap A=\emptyset$, which in other words means that $p(\vvx):=q(\vvx)+\xi$ does not attain values in  $A$ for $\vvx \in K$. Our conclusion follows from \eqref{aaa1}, \eqref{aaa2} and the triangle inequality.
 \end{proof}

  \begin{thm} \label{TH5}
  Let  $K \subset \R^n$  and let $A \subset \R^m$. Furthermore assume that
   \begin{gather}
     \overline{\dim}_{B}(K)+\dim_{H}(A)<m, \\ \intertext{or}
      \dim_{H}(K)+\overline{\dim_{B}}(A)<m.
        \end{gather}
        Then given any continuous function $f:K \to \R^m$ and any $\varepsilon>0$ there exists some polynomial mapping $p$ such that $p(x) \not \in A$ if $x \in K$ and such that
  \[
    \max_{\vvx \in K} \norm{f(\vvx)-p(\vvx)}<\varepsilon.
  \]
  \end{thm}
  \begin{proof}
  (We follow the proof of Theorem \ref{TH4}) By the multivariate Weierstrass theorem we may find a polynomial mapping $q: \R^n \to \R^m$ such that
\begin{gather} \label{aaaa1} \max_{\vvx \in K}\norm{q(\vvx)-f (\vvx)}<\frac {\varepsilon} 2 
\end{gather}
   Since a polynomial map $q:\R^n \to \R^m$  is Lipschitz we have that $\dim_H(q(K)) \leq \dim_H(K)$ and $\overline{\dim}_B(q(K)) \leq \overline{\dim}_B(K)$. Thus by Proposition \ref{prop2} we have that  for some $\xi \in \R^m$ with
   \begin{gather} \label{aaaa2} \norm{\xi}<\frac \varepsilon 2 \end{gather} that $(q(K)+\xi) \cap A=\emptyset$, which in other words means that $p(\vvx):=q(\vvx)+\xi$ does not attain values in  $A$ for $\vvx \in K$. Our conclusion follows from \eqref{aaaa1}, \eqref{aaaa2} and the triangle inequality.
 \end{proof}
 
In several complex variables Lavrentiev's theorem may be generalised to a result of Harvey-Wells\footnote{This  is a sharper version of a result of Hörmander-Wermer \cite{HWe}, see discussion in \cite[section 8.]{Levenberg}.}
 \cite{HW} that if $E$  is a {\em totally real\footnote{For the definition see \cite[p.115]{Levenberg}.}}  submanifold of class $C^1$ in an open set in $\C^n$ then for any compact $K \subset E$ where $K$ is polynomially convex then any function $f$ continuous on $K$ may be uniformly approximated by polynomials on $K$. We remark that just like the classical Weierstrass theorem is a special case of the Lavrentiev theorem (since an interval $[a,b] \in \R$ is a compact set with connected complement and without interior points in $\C$), the multivariate Weierstrass theorem is a special case of the Harvey-Wells theorem (since $\R^n$ is a totally real manifold in $\C^{n}$ and any compact set $K \subset \R^n$ is polynomially convex).

\begin{thm} \label{TH6}
 Let $E \subset \C^n$ be a totally real manifold of class $\mathcal C^1$ and let  $K \subset E$ be a compact set that is polynomially convex. Then for any continuous function $f:K \to \C^m$ and any set $A \subset \C^m$ with $\overline{\dim}_B (K) +\dim_H (A)<2m$ or ${\dim}_H(K)  +\overline{\dim}_B (A)<2m$
there exists some polynomial mapping $p$ such that \[
    \max_{z \in K} \abs{f(z)-p(z)}<\varepsilon.
  \] and such that $p(z) \not \in A$ if $z \in K$.
\end{thm}
\begin{proof}
  (We follow the proof of Theorem \ref{TH5}) By the Harvey-Wells theorem we may find a polynomial map $q: \C^n \to \C^m$ such that
\begin{gather} \label{aaaaa1} \max_{\vvx \in K}\norm{q(\vvx)-f (\vvx)}<\frac {\varepsilon} 2.
\end{gather}
Thus by identifying $\C^m$ with $\R^{2m}$ and by Proposition \ref{prop2} we have that  for some $\xi \in \C^m$ with
   \begin{gather} \label{aaaaa2} \norm{\xi}<\frac \varepsilon 2 \end{gather} that $(q(K)+\xi) \cap A=\emptyset$, which in other words means that $p(\vvx):=q(\vvx)+\xi$ does not attain values in  $A$ for $\vvx \in K$. Our conclusion follows from \eqref{aaaaa1}, \eqref{aaaaa2} and the triangle inequality.
 \end{proof}
 
 Since $E \subset \C^n$ is a totally real manifold of class $\mathcal C^1$ in particular means that $\dim_H (E)=\overline{\dim}_B(E) \leq n $,  a sufficient condition in Theorem \ref{TH6} is that $\dim_H A<2m-n$. We obtain the following Corollary.

\begin{cor} \label{cor}
 Let $E \subset \C^n$ be a totally real manifold of class $\mathcal C^1$ and let  $K \subset E$ be a compact set that is polynomially convex. Then for any continuous function $f:K \to \C^m$ and any set $A \subset \C^m$ with $\dim_H (A)<2m-n$
there exists some polynomial function $p$ such that \[
    \max_{z \in K} \abs{f(z)-p(z)}<\varepsilon.
  \] and such that $p(z) \not \in A$ if $z \in K$.
\end{cor}

\subsection{Infinite dimensional results}
It is well known that Weierstrass theorem on approximation of continuous functions on compact sets by polynomials generalizes from $\R^n$ to infinite dimensional real Banach spaces  \cite{ww}.  Since compact sets in infinite-dimensional Banach spaces are always without interior points the analogue of Lemma \ref{LE4} become trivial.
We are thus ready to prove our main approximation theorem on Banach spaces wich is a natural analogue of Theorem \ref{TH1} and Theorem \ref{TH6}.

\begin{thm}  \label{TH7}
 Let $\mathcal B$ be an infinite dimensional real Banach space, let $A \subset \mathcal B$ be a countable set, let $K \subset \mathcal B$ be a compact set and let $f:K \to  \mathcal B$ be a continuous function.
  Then for any $\varepsilon>0$  there exists some polynomial mapping $p:K \to \mathcal B$ such that
 \begin{gather*}
   \max_{x \in K} \norm{f(x)-p(x)}<\varepsilon,
 \end{gather*}
 and such that $p(x) \not \in  A$ if $x \in K$.
   \end{thm}

   \begin{proof}
  (We follow the proof of Theorem \ref{TH5}) By the Weierstrass approximation theorem for real Banach spaces \cite[Theorem 2.5]{ww} we may find a continuous polynomial map $q: \mathcal B \to \mathcal B$ such that
\begin{gather} \label{kk1} \max_{\vvx \in K}\norm{q(\vvx)-f (\vvx)}<\frac {\varepsilon} 2.
\end{gather}
Since $K$ is compact and  $q$ is continuous, also $q(K)$ is compact.
Since $\mathcal B$ is infinite dimensional and any compact set in an infinite dimensional Banach space is a compact set without interior points, it follows that $q(K)$ is a compact set without interior points. By  Proposition \ref{prop3} we have that  for some $\xi \in \mathcal B$ with
   \begin{gather} \label{kk2} \norm{\xi}<\frac \varepsilon 2 \end{gather} that $(q(K)+\xi) \cap A=\emptyset$, which in other words means that $p(\vvx):=q(\vvx)+\xi$ does not attain values in  $A$ for $\vvx \in K$. Our conclusion follows from \eqref{kk1}, \eqref{kk2} and the triangle inequality.
 \end{proof}

\bibliographystyle{plain}

\end{document}